\documentclass{amsart}
\usepackage{amsmath,amssymb,pstricks}

\author[J.\ Gong]{Jasun Gong}
\address{
Department of Mathematics,
  University of Pittsburgh,
  Pittsburgh, PA 15260}
\email{jasun@pitt.edu}
\title[Extension Theorem for Homeomorphisms of Class $\LW$]{A Schoenflies Extension Theorem for a Class of Locally bi-Lipschitz Homeomorphisms}
\date{Tuesday, 24 November 2009}

\theoremstyle{plain}
\newtheorem{thm}{Theorem}[section]

\newtheorem{cor}[thm]{Corollary}
\newtheorem{lemma}[thm]{Lemma}

\theoremstyle{definition}
\newtheorem{defn}[thm]{Definition}
\newtheorem{rmk}[thm]{Remark}
\newtheorem{claim}[thm]{Claim}
\newtheorem{ques}[thm]{Question}

\numberwithin{equation}{section}

\newcommand{\B}{\mathbb{B}}
\newcommand{\del}{\partial}
\newcommand{\dist}{\operatorname{dist}}
\newcommand{\e}{\epsilon}
\newcommand{\id}{\operatorname{id}}
\newcommand{\loc}{\textrm{loc}}
\newcommand{\LW}{LW^p_2}
\newcommand{\N}{\mathbb{N}}
\renewcommand{\S}{\mathbb{S}}

\newcommand{\R}{\mathbb{R}}

\begin{document}

\begin{abstract}
In this paper we prove a new version of the Schoenflies extension theorem for collared domains $\Omega$ and $\Omega'$ in $\R^n$: for $p \in [1,n)$, locally bi-Lipschitz homeomorphisms from $\Omega$ to $\Omega'$ with locally $p$-integrable, second-order weak derivatives admit homeomorphic extensions of the same regularity.

Moreover, the theorem is essentially sharp.  The existence of exotic $7$-spheres shows that such extension theorems cannot hold, for $p > n = 7$.
\end{abstract}

\maketitle

\section{Introduction}

\subsection{Embeddings of Collars}

In point-set topology, the Schoenflies Theorem \cite[Thm III.5.9]{Wilder} is a stronger form of the well-known Jordan Curve Theorem:\ it states that
\emph{every simple closed curve separates the sphere $\S^2$ into two domains, each of which is homeomorphic to $\B^2$, the open unit disc}.  The same statement does not hold in higher dimensions, since the Alexander horned sphere \cite{Alexanderhorned} provides a counter-example in $\R^3$.  Despite this, Brown \cite{Brown} proved that for each $n \in \N$, every embedding of $\S^{n-1} \times (-\e,\e)$ into $\R^n$ extends to an embedding of $\B^n$ into $\R^n$.

Similar extension problems arise by varying the regularity of the embeddings.  To this end, we prove a Schoenflies-type theorem for a new class of homeomorphisms.
Their regularity is given in terms of Sobolev spaces and Lipschitz continuity.  

To begin, recall that a homeomorphism $f : \Omega \to \Omega'$ is \emph{locally bi-Lipschitz} if for each $z \in \Omega$, there is a neighborhood $O$ of $z$ and $L \geq 1$ so that the inequality 
\begin{equation} \label{eq_bilip}
L^{-1}\,|x-y| \;\leq\; |f(x) - f(y)| \;\leq\; L\,|x-y|
\end{equation}
holds for all $x,y \in O$.
Recall also that for $p \geq 1$ and $k \in \N$, the Sobolev space $W^{k,p}_\loc(\Omega;\Omega')$ consists of maps $f : \Omega \to \Omega'$, where each component $f_i$ lies in $L^p_\loc(\Omega)$ and has weak derivatives of orders up to $k$ in $L^p_\loc(\Omega)$.

\begin{defn}
Let $f : \Omega \to \Omega'$ be a locally bi-Lipschitz homeomorphism.  For $p \in [1, \infty)$, we say that $f$ is of \emph{class $\LW$} if $f \in W^{2,p}_\loc(\Omega;\Omega')$ and $f^{-1} \in W^{2,p}_\loc(\Omega';\Omega)$.  If $K$ and $K'$ are closed sets, a homeomorphism $f : K \to K'$ is of class $\LW$ if the restriction of $f$ to the interior of $K$ is of class $\LW$.
\end{defn}

Instead of product sets of the form $\S^{n-1} \times (-\e,\e)$, we will consider domains in $\R^n$ of a similar topological type.

\begin{defn}
A bounded domain $D$ in $\R^n_*$ is \emph{Jordan} if its boundary $\partial D$ is homeomorphic to $\S^{n-1}$.  A \emph{collared domain} (or \emph{collar}) is a domain in $\R^n$ of the form $D_2 \setminus \bar{D}_1$, where $D_1$ and $D_2$ are Jordan domains with $\bar{D}_1 \subset D_2$.
\end{defn}

We now state the extension theorem for homeomorphisms of class $\LW$ between collared domains.

\begin{thm} \label{thm_mainthm}
Let $D_1$ and $D_2$ be Jordan domains in $\R^n$ so that $\bar{D}_1 \subset D_2$, let $B_1$ and $B_2$ be balls so that $\bar{B}_1 \subset B_2$, and let $p \in [1,n)$.

If $f : \bar{D}_2 \setminus D_1 \to \bar{B}_2 \setminus B_1$ is a homeomorphism of class $\LW$ so that $f(\partial D_i) = \partial B_i$ holds, for $i = 1, 2$, then there exists a homeomorphism $F : \bar{D}_2 \to \bar{B}_2$ of class $\LW$ and a neighborhood $N$ of $\partial D_2$ so that
$F|(N \cap \bar{D}_2) = f|(N \cap \bar{D}_2).$
\end{thm}

The proof 
is an adaptation of Gehring's argument \cite[Thm 2']{Gehring} from the class of quasiconformal homeomorphisms to the class $\LW$.  For the locally bi-Lipschitz class, the extension theorem was known to Sullivan \cite{Sullivan} and later proved by Tukia and V\"{a}is\"{a}l\"{a} \cite[Thm 5.10]{TukiaVaisala}.  For more about quasiconformal homeomorphisms, see \cite{Vaisala}.

As in Gehring's case, Theorem \ref{thm_mainthm} is not quantitative.  His extension depends on the distortion (resp.\ Lipschitz constants) of $g$ as well as the configurations of the collars $D_2 \setminus \bar{D}_1$ and $B_2 \setminus \bar{B}_1$.  In addition, our modification of his extension also depends explicitly on the parameters $p$ and $n$.

\subsection{Motivations, Smoothness, and Sharpness} 

The motivation for Theorem \ref{thm_mainthm} comes from the study of Lipschitz manifolds.
Specifically, Heinonen and Keith have recently shown that \emph{if an $n$-dimensional Lipschitz manifold, for $n \neq 4$, admits an atlas with coordinate charts in the Sobolev class $W^{2,2}_{\loc}(\R^n;\R^n)$, then it admits a smooth structure} \cite{HeinonenKeith}.  

On the other hand, there are $10$-dimensional Lipschitz manifolds without smooth structures \cite{Kervaire}.  This leads to the following question:

\begin{ques}
For $n \neq 4$, does there exist $p \in [1,2)$ so that every $n$-dimensional Lipschitz manifold admits an atlas of charts in $W^{2,p}_{loc}(\R^n;\R^n)$?
\end{ques}  

Sullivan has shown that \emph{every $n$-dimensional topological manifold, for $n \neq 4$, admits a Lipschitz structure} \cite{Sullivan}.  
A key step in the proof is to show that bi-Lipschitz homeomorphisms satisfy a Schoenflies-type extension theorem.  
One may inquire whether this direction of proof would also lead to the desired Sobolev regularity.
Theorem \ref{thm_mainthm} would be a first step in this direction.  For more about Lipschitz structures on manifolds, see \cite{LuukkainenVaisala}.

It is worth noting that Theorem \ref{thm_mainthm} is not generally true for $p > n$.
Recall that for any domain $\Omega$ in $\R^n$, Morrey's inequality \cite[Thm 4.5.3.3]{EvansGariepy} gives $W^{2,p}(\Omega) \hookrightarrow C^{1,1 - n/p}(\Omega)$, so homeomorphisms of class $\LW$ are necessarily $C^1$-diffeomorphisms.

Indeed, every $C^\infty$-diffeomorphism $\varphi : \S^{n-1} \to \S^{n-1}$ admits a radial extension
$$
\bar\varphi(x) \;:=\; |x| \, \varphi \Big( \frac{x}{|x|} \Big)
$$ 
that is, a $C^\infty$-diffeomorphism between round annuli. The validity of Theorem \ref{thm_mainthm}, for $p > n$, would therefore imply that every such $\varphi$ extends to a $C^1$-diffeomorphism of $\bar\B^n$ onto itself.  However, for $n=7$ this conclusion is impossible.	

Recall that every such $\varphi$ also determines a $C^\infty$-smooth, $n$-dimensional manifold $M_\varphi^n$ that is homeomorphic to $\S^n$ \cite[Construction (C)]{Milnor}.  Indeed, $M_\varphi^n$ is the quotient of two copies of $\R^n$ under the relation $x \sim \varphi^*(x)$ on $\R^n \setminus \{0\}$, where
\begin{equation} \label{eq_twistdiffeo}
\varphi^*(x) \;:=\; \frac{1}{|x|} \, \varphi\Big( \frac{x}{|x|} \Big).
\end{equation}
If $\varphi$ is the identity map on $\S^{n-1}$, then $\varphi^*$ is the inversion map $x \mapsto |x|^{-2}x$, and $M_\varphi^n$ is precisely $\S^n$.
By using invariants from differential topology, Milnor proved the following theorem about such manifolds \cite[Thm 3]{Milnor}.

\begin{thm}[Milnor, 1956]
There exist $C^\infty$-smooth manifolds of the form $M_\varphi^7$ that are homeomorphic, but not $C^\infty$-diffeomorphic, to $\S^7$.
\end{thm}

Such manifolds are better known as \emph{exotic spheres}.  The next lemma is an analogue of \cite[Thm 8.2.1]{Hirsch}; it relates exotic spheres to extension theorems.  

\begin{lemma} \label{lemma_extdiffeo}
Let $\varphi : \S^{n-1} \to \S^{n-1}$ be a $C^\infty$-diffeomorphism and let $\bar\varphi : \bar\B^n \setminus \{0\} \to \bar\B^n \setminus \{0\}$ be its radial (diffeomorphic) extension.  If there exists a $C^1$-diffeomorphism $\Phi : \bar\B^n \to \bar\B^n$ that agrees with $\bar\varphi$ on a neighborhood of $\S^{n-1}$ in $\bar\B^n$, then $M_\varphi^n$ is $C^1$-diffeomorphic to $\S^n$.
\end{lemma}

\begin{proof}[Proof of Lemma \ref{lemma_extdiffeo}]
Let $\varphi^*$ be the diffeomorphism defined in Equation \eqref{eq_twistdiffeo}.  By construction, there is an atlas of charts $\{M_i\}_{i=1}^2$ for $M_\varphi^n$ with homeomorphisms $\psi_i : M_i \to \R^n$ that satisfy $\psi_1 \circ \psi_2^{-1} = \varphi^*$.

Let $\pi_1, \pi_2 : \R^n \to \S^n$ be stereographic projections relative to the ``north'' and ``south'' poles on $\S^n$, respectively, so 
$\pi_2^{-1} \circ \pi_1 = \id^* = (\id^*)^{-1}$.  
Observe that	
$$
((\id^*)^{-1} \circ \, \varphi^*)(x) \;=\;
\frac{\varphi^*(x)}{|\varphi^*(x)|^2} \;=\;
|x| \, \varphi\Big( \frac{x}{|x|} \Big) \;=\; 
\bar\varphi(x)
$$
holds for all $x \in \R^n \setminus \{0\}$.  It follows that
$$
x \;\mapsto\;
\left\{\begin{array}{ll}
(\pi_1^{-1} \circ \psi_1)(x), & \textrm{if } x \in M_1 \\
(\pi_2^{-1} \circ \Phi \circ \psi_2)(x), & \textrm{if } x \in M_2
\end{array}
\right.
$$
is a $C^1$-diffeomorphism of $M_\varphi^n$ onto $\S^n$.
\end{proof}

By \cite[Thm 2.2.10]{Hirsch}, if two $C^\infty$-smooth manifolds are $C^1$-diffeomorphic, then they are $C^\infty$-diffeomorphic.  It follows that there exist $C^1$-diffeomorphisms of collars in $\R^7$ that do not admit diffeomorphic extensions of class $\LW$, for any $p > 7$.

The next result follows from the inclusion $W^{2,p}_{loc}(\Omega;\Omega') \subseteq W^{2,q}_{loc}(\Omega;\Omega')$, for $q \leq p$.

\begin{cor}
Let $n = 7$.  For $p > n$, there exist collars $\Omega$, $\Omega'$ in $\R^n$ and homeomorphisms $\varphi : \Omega \to \Omega'$ of class $\LW$ that admit homeomorphic extensions of class $LW^q_2$, for every $1 \leq q < n$, but not of class $\LW$.
\end{cor}

Since the above discussion relies crucially on Sobolev embedding theorems, it leaves open the borderline case $p=n$.

\begin{ques}
Is Theorem \ref{thm_mainthm} true for the case $p = n$?
\end{ques}

For $p > n$, the main obstruction to an extension theorem is the existence of exotic $n$-spheres.  It is known that no exotic spheres exist for $n = 1, 2, 3, 5, 6$ \cite{KervaireMilnor}, and the case $n=1$ can be done by hand.
It would be interesting to determine whether other geometric obstructions arise.

\begin{ques}
For $n = 2, 3, 5, 6$, is Theorem \ref{thm_mainthm} true for all $p \geq 1$?
\end{ques}

The outline of the paper is as follows.  In Section 2 we review basic facts about Lipschitz mappings, Sobolev spaces, and the class $\LW$.  In Section 3 we prove extension theorems in the setting of doubly-punctured domains.  Section 4 addresses the case of homeomorphisms between collars, by employing suitable generalizations of inversion maps and reducing to previous cases.

\subsection{Acknowledgments}
The author is especially indebted to his late advisor and teacher, Juha
Heinonen, for numerous insightful discussions and for directing him in this area
of research.  He thanks Piotr Haj{\l}asz for many helpful conversations, which
led to key improvements of this work.   He also thanks Leonid Kovalev, Jani
Onninen, Pekka Pankka, Mikko Parviainen, and Axel Straschnoy for their helpful comments and suggestions on a preliminary version of this work.

The author acknowledges the kind hospitality of the University of Michigan and the Universitat Aut\`onoma de Barcelona, where parts of this paper were written.  This project was partially supported by NSF grant DMS-0602191.

\section{Notation and Basic Facts}

For $A \subset \R^n$, we write $A^c$ for the complement of $A$ in $\R^n$.  The open unit ball in $\R^n$ is denoted $\B^n$; if the dimension is understood, we will write $\B$ for $\B^n$.

We write $A \lesssim B$ for inequalities of the form $A \leq k B$, where $k$ is a fixed dimensional constant and does not depend on $A$ or $B$.

For domains $\Omega$ and $\Omega'$ in $\R^n$, recall that a map $f : \Omega \to \Omega'$ is \emph{Lipschitz} whenever
$$
L(f) \;:=\: 
\sup\left\{ \frac{ |f(x)-f(y)| }{ |x-y| } \,:\, x,y \in \Omega,\, x \neq y \right\} \;<\; \infty.
$$
The map $f$ is \emph{locally Lipschitz} if every point in $\Omega$ has a neighborhood on which $f$ is Lipschitz.  A homeomorphism $f : \Omega \to \Omega'$ is \emph{bi-Lipschitz} (resp.\ \emph{locally bi-Lipschitz}) if $f$ and $f^{-1}$ are both Lipschitz (resp.\ locally Lipschitz); compare Equation \eqref{eq_bilip}.

The following lemmas about bi-Lipschitz maps are used in Section 2.  The first is a special case of \cite[Lemma 2.17]{TukiaVaisala}; the second one is elementary, so we omit the proof.

\begin{lemma}[Tukia-V{\"a}is{\"a}l{\"a}] \label{lemma_loclip}
Let $O$ and $O'$ be open, connected sets in $\R^n$ and let $K$ be a compact subset of $O$.  If $f : O \to O'$ is locally bi-Lipschitz, then $f|K$ is bi-Lipschitz, where $L((f|K)^{-1})$ depends only on $O$, $K$, and $L(f)$.
\end{lemma}

\begin{lemma} \label{lemma_gluelip}
For $i = 1, 2$, let $h_i : \Omega_i \to \R^n$ be locally bi-Lipschitz embeddings so that $h_1(\Omega_1 \setminus \Omega_2) \cap h_2(\Omega_2 \setminus \Omega_1) = \emptyset$.  If $h_1 = h_2$ holds on all of $\Omega_1 \cap \Omega_2$, then
$$
h(x) \;=\; 
\left\{
\begin{array}{ll}
h_1(x), & \textrm{ if } x \in \Omega_1 \\
h_2(x), & \textrm{ if } x \in \Omega_2 \setminus \Omega_1
\end{array}
\right.
$$
is also a locally bi-Lipschitz embedding.
\end{lemma}

For $f \in W^{2,p}(\Omega;\Omega')$, we will use the Hilbert-Schmidt norm for the 
weak derivatives $Df(x) := [\del_jf_i(x)]_{i,1=1}^n$ and $D^2f(x) := [\del_k\del_jf_i(x)]_{i,j,k=1}^n$.  That is,
$$
|Df(x)| \;:=\; \Big[ \sum_{i,j=1}^n |\del_jf_i(x)|^2 \Big]^{1/2}, \; \;
|D^2f(x)| \;:=\; \Big[ \sum_{i,j,k=1}^n |\del_k\del_jf_i(x)|^2 \Big]^{1/2}.
$$

In what follows, we will use basic facts about Sobolev spaces, such as the change of variables formula \cite[Thm 2.2.2]{Ziemer} and that Lipschitz functions on $\Omega$ are characterized by the class $W^{1,\infty}(\Omega)$ \cite[Thm 4.2.3.5]{EvansGariepy}.  The lemma below gives a gluing procedure for Sobolev functions.

\begin{lemma} \label{lemma_gluing}
For $i=1,2$, let $O_i$ be a domain in $\R^n$ and let $f_i \in W^{1,p}_{loc}(O_i)$.  If $f_1 = f_2$ holds a.e.\ on $O_1 \cap O_2$, then $\chi_{O_1}f_1 + \chi_{O_2 \setminus O_1}f_2 \in W^{1,p}_{loc}(O_1 \cup O_2)$.
\end{lemma}

\begin{proof}
Let $O$ be a bounded domain in $\R^n$ so that $\bar{O} \subset O_1 \cup O_2$.  For each $x \in O$, there exists $r > 0$ so that $B(x,r)$ lies entirely in $O_1$ or in $O_2$.  Since $\bar{O}$ is compact, there exists $N \in \N$ and a collection of balls $\{B(x_i,r_i)\}_{i=1}^N$ whose union covers $O$.

Let $\{\varphi_i\}_{i=1}^N$ be a smooth partition of unity that is subordinate to the cover $\{B(x_i,r_i)\}_{i=1}^N$.  For each $i =1, 2, \ldots, N$, one of $f_1\varphi_i$ or $f_2\varphi_i$ is well-defined and lies in $W^{1,p}(O)$; call it $\psi_i$.  We now observe that $\psi := \sum_{i=1}^N \psi_i$ also lies in $W^{1,p}(O)$ and by construction, it agrees with $\chi_{O_1}f_1 + \chi_{O_2 \setminus O_1}f_2$.
\end{proof}

It is a fact that the class $\LW$ is preserved under composition.  This is stated as a lemma below, and it follows directly from the product rule \cite[Thm 4.2.2.4]{EvansGariepy} and the change of variables formula \cite[Thm 2.2.2]{Ziemer}.

\begin{lemma} \label{lemma_changevari}
Let $p \geq 1$.  If $f : \Omega \to \Omega'$ and $g : \Omega' \to \Omega''$ are homeomorphisms of class $\LW$, then so is $h := g \circ f$. 
In addition, for a.e.\ $x \in \Omega$ and for all $i,j,k \in \{1, \cdots n\}$, the weak derivatives satisfy
\begin{equation}
\label{eq_changevari}
\left\{\begin{split}
\del_jh_i(x) &\;=\;
\sum_{l=1}^n \del_lg_i(f(x)) \del_jf_l(x) \\
\del^2_{kj}h_i(x) &\;=\;
\sum_{l=1}^n \Big[ 
\del_lg_i(f(x)) \del^2_{kj}f_l(x) + 
\sum_{m=1}^n \del^2_{ml}g_i(f(x)) \del_kf_m(x) \del_jf_l(x)
\Big].
\end{split}\right.
\end{equation}
\end{lemma}

\begin{rmk}
Linear maps (homeomorphisms) such as dilation and translation, are clearly of class $\LW$.  So if $g: \Omega \to \Omega'$ is any homeomorphism of class $\LW$, then by Lemma \ref{lemma_changevari}, its composition with such linear maps is also of class $\LW$.  In what follows, we will implicitly use this fact to obtain convenient geometrical configurations.
\end{rmk}

\section{Extensions for Homeomorphisms of Class $\LW$ between Doubly-Punctured Domains}

First we formulate the extension theorem in a different geometric configuration.

\begin{thm} 
\label{thm_disjointcase}
Let $p \geq 1$, let $E_1$ and $E_2$ be Jordan domains so that $\overline{E_1} \cap \overline{E_2} = \emptyset$, and let $B_1$ and $B_2$ be balls so that $\overline{B_1} \cap \overline{B_2} = \emptyset$.

If $g : (E_2 \cup E_1)^c \to (B_1 \cup B_2)^c$ is a homeomorphism of class $\LW$ so that $g(\del E_i) = \del B_i$ holds, for $i=1,2$, then there exists a homeomorphism $G : E_2^c \to B_2^c$ of class $\LW$ and a neighborhood $N$ of $\del E_2$ so that $g|(N \cap E_2^c) = G|(N \cap E_2^c)$.  

\end{thm}

Following the outline of \cite[Sect 3]{Gehring}, we begin with a special case.

\begin{lemma} \label{lemma_identity}
Theorem \ref{thm_disjointcase} holds under the additional assumption that 
\begin{equation} \label{eq_identityball}
g|B^c \;=\; \id|B^c
\end{equation}
where $B$ is an open ball that contains $\bar{E_1}$ and $\bar{E_2}$.
\end{lemma}

\begin{proof}
\emph{Step 1}.\
By composing with linear maps, we may assume that $B = \B$, and that there exist $a,b \in \R$ so that $a < b$ and $\bar{B}_1 \subset \{x_n < a\}$ and $\bar{B}_2 \subset \{x_n > b\}$.

Put $c = (b-a)/2$.  Define an odd, $C^{1,1}$-smooth function $s_0 : \R \to [-1,1]$ by
$$
s_0(t) \;:=\;
\left\{\begin{array}{ll}
1- (t- c)^2/c^2, & \textrm{if } 0 \leq t \leq c \\
1, & \textrm{if } t > c
\end{array}\right.
$$
and using the auxiliary function $s : \R \to [0,3]$, given by
$$
s(t) \;:=\; \frac{3}{2}\Big( s_0\Big(t - \frac{a+b}{2}\Big) + 1 \Big)
$$
we define a bi-Lipschitz homeomorphism $S : \R^n \to \R^n$ by
\begin{equation} \label{eq_shearmap}
S(x) \;=\; x - s(x_n) \, e_1.
\end{equation}
It is clear that $S$ is 
of class $LW^p$ and satisfies the a.e.\ estimate
\begin{equation} 
\label{eq_derivSbd}
|D^2S| \;\leq\;
2c^{-2}.
\end{equation}

\begin{figure}[h]
\centering
\begin{pspicture*}(-6,-1.25)(6,1.25)
\psarc{<-}(0,0){.75}{30}{150}
\put(0,.25){$S$}
\multiput(0.25,0)(.5,0){4}{ \psline[linewidth=2pt]{<->}(-3.5,-1.25)(-3.5,1.25) }
\multiput(0.25,0)(.5,0){4}{ \psbezier[linewidth=2pt](2,.35)(2,0)(3.5,0)(3.5,-.35) }
\multiput(0.25,0)(.5,0){4}{ \psline[linewidth=2pt]{<-}(2,1.25)(2,.35) }
\multiput(0.25,0)(.5,0){4}{ \psline[linewidth=2pt]{->}(3.5,-.35)(3.5,-1.25) }
\multiput(0,0)(7,0){2}{ \psline{->}(-5,-1.25)(-5,1.25) }
\multiput(0,0)(7,0){2}{ \psline{->}(-5.5,-.75)(-1.5,-.75) }
\multiput(0,0)(7,0){2}{ \put(-1.4,-.8){$x_1$} }
\multiput(0,0)(7,0){2}{ \put(-5.5,1.1){$x_n$} }
\multiput(0,0)(7,0){2}{ \psline[linestyle=dashed](-5.5,-.35)(-1.5,-.35) }
\multiput(0,0)(7,0){2}{ \psline[linestyle=dashed](-5.5,.35)(-1.5,.35) }
\multiput(0,0)(7,0){2}{ \put(-5.25,-.3){$a$} }
\multiput(0,0)(7,0){2}{ \put(-5.25,.4){$b$} }
\end{pspicture*}
\caption{For $\R^2$, level curves for the map $S$.}
\end{figure}
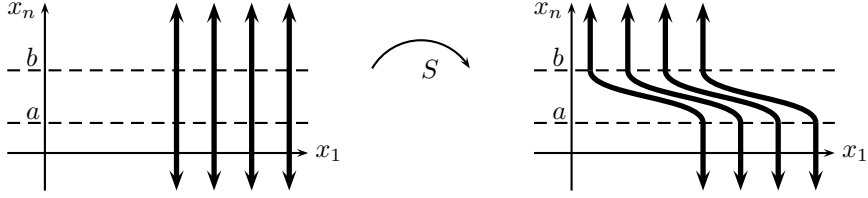

\emph{Step 2}.\
For $k \in \mathbb{Z}$, put $\tau_k(x) = x + 3k e_1$ and consider the sets
$$
\Omega \;:=\; \Big( \bigcup_{k=0}^\infty \tau_k(E_1) \cup \tau_k(E_2) \Big)^c
\textrm{ and }
\Omega' \;:=\; \Big( \bigcup_{k=0}^\infty \tau_k(B_1) \cup \tau_k(B_2) \Big)^c. 
$$
We now modify $g$ into a new homeomorphism $g_* : \Omega \to \Omega'$, as follows:
\begin{equation} \label{eq_periodicmap}
g_*(x) \;:=\;
\left\{\begin{array}{ll}
(\tau_k \circ g \circ \tau_{-k})(x), &
\textrm{if } x \in \Omega \cap \tau_k(\B), 
\textrm{ for some } k \geq 0 \\
x, & 
\textrm{if } x \in \Omega \setminus \bigcup_{k=0}^\infty \tau_k(\B).
\end{array}\right.
\end{equation}
By our hypotheses, there exists $r \in (0,1)$ so that $E_1 \cup E_2 \subset B(0,r)$ and so that $g|\B \setminus B(0,r) = \id$.  Putting $\Omega_1 := \tau_k(\B) \cap \Omega$ and $\Omega_2 := \Omega \setminus \bigcup_{l=0}^\infty \tau_k(\overline{B(0,r)})$ for each $k \in \N$, Lemma \ref{lemma_gluelip} implies that $g_*$ is locally bi-Lipschitz.

Similarly, for any bounded domain $O$ in $\Omega$ that meets $\tau_k(\partial \B)$, put $O_1 := O \cap \Omega$ and $O_2 := O \setminus \tau_k(\overline{B(0,r)})$.  For $f_1 := D(\tau_k \circ g \circ \tau_{-k})$ and $f_2 := D(\id)$, Lemma \ref{lemma_gluing} implies that $g_* \in W^{2,p}(O)$  and therefore $g_* \in W^{2,p}_{loc}(\Omega;\Omega')$.  By symmetry, the same is true of $g_*^{-1}$, so $g_*$ is of class $\LW$.

\emph{Step 3}.\ Consider the bi-Lipschitz homeomorphism given by
\begin{equation} \label{eq_periodicext}
G_* \;:=\; \tau_1 \circ g_*^{-1} \circ S \circ g_*.
\end{equation}

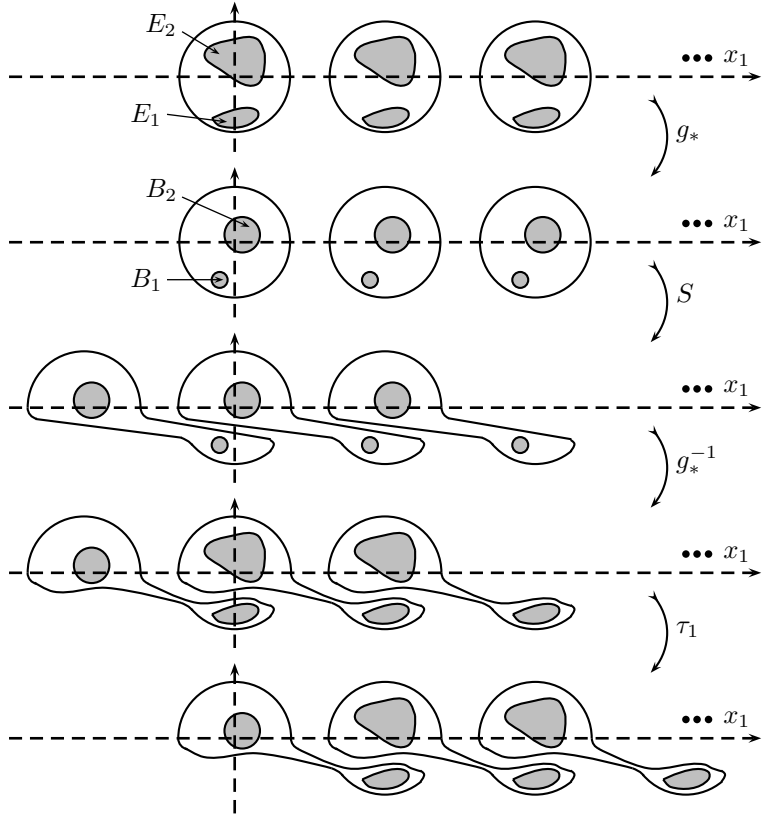
\begin{figure}[h] \label{fig_gehring}
\centering
\begin{pspicture*}(-3,-9.8)(7,1)
\multiput(0,.3)(2,0){3}{
\pscurve[fillcolor=lightgray,fillstyle=solid,framearc=.2](.4,0)(.3,.2)(0,.2)(-.4,0)(-.2,-.2)(.3,-.4)(.4,0) }
\multiput(-.1,-.15)(2,0){3}{
\pscurve[fillcolor=lightgray,fillstyle=solid,framearc=.2](-.2,-.4)(.4,-.3)(.4,-.4)(-.1,-.5)(-.2,-.4) }
%
\multiput(0,-2.5)(2,0){3}{ \pscircle[fillcolor=lightgray,fillstyle=solid,framearc=.2](.1,.4){.25} }
\multiput(0,-2.3)(2,0){3}{ \pscircle[fillcolor=lightgray,fillstyle=solid,framearc=.2](-.2,-.4){.12} }
%
\multiput(-2,-4.7)(2,0){3}{ \pscircle[fillcolor=lightgray,fillstyle=solid,framearc=.2](.1,.4){.25} }
\multiput(0,-4.5)(2,0){3}{ \pscircle[fillcolor=lightgray,fillstyle=solid,framearc=.2](-.2,-.4){.12} }
%
\multiput(0,0)(2,-2.2){2}{ \put(-2,-6.9){ \pscircle[fillcolor=lightgray,fillstyle=solid,framearc=.2](.1,.4){.25} } }
\multiput(0,0)(2,-2.2){2}{ \multiput(0,-6.3)(2,0){2}{
\pscurve[fillcolor=lightgray,fillstyle=solid,framearc=.2](.4,0)(.3,.2)(0,.2)(-.4,0)(-.2,-.2)(.3,-.4)(.4,0) } }
\multiput(.1,-.05)(2,-2.2){2}{ \multiput(-.2,-6.7)(2,0){3}{
\pscurve[fillcolor=lightgray,fillstyle=solid,framearc=.2](-.2,-.4)(.4,-.3)(.4,-.4)(-.1,-.5)(-.2,-.4) } }
%
\multiput(0,0)(0,-2.2){5}{ \psline[linewidth=1pt,linestyle=dashed]{->}(-3,0)(7,0) }
\multiput(0,0)(0,-2.2){5}{ \psline[linewidth=1pt,linestyle=dashed]{->}(0,-1)(0,1) }
\multiput(0,0)(0,-2.2){5}{ \put(6.5,.2){$x_1$} }
\multiput(0,0)(0,-2.2){5}{ \multiput(6,.25)(.15,0){3}{ \pscircle*(0,0){.05} } }
\multiput(0,0)(0,-2.2){2}{ \multiput(0,0)(2,0){3}{ \pscircle(0,0){.75} } }
\multiput(0,-4.4)(0,-2.2){2}{ \multiput(0,0)(2,0){3}{ \psarc(0,0){.75}{-140}{-50} } }
\multiput(-2,-4.4)(0,-2.2){2}{ \multiput(0,0)(2,0){3}{ \psarc(0,0){.75}{0}{180} } }
\multiput(-2.6,-4.4)(0,-2.2){2}{ \multiput(0,0)(2,0){3}{ \psarc(0,0){.15}{180}{260} } }
\multiput(-1.1,-4.4)(0,-2.2){2}{ \multiput(0,0)(2,0){3}{ \psarc(0,0){.15}{180}{250} } }
\multiput(-.56,-4.9)(0,-2.2){2}{ \multiput(0,0)(2,0){3}{ \pscurve(0,0)(-.05,.05)(-.15,.1) } }
\multiput(0,-4.4)(2,0){3}{ \psline(-.7,-.4)(-2.65,-.15) }
\multiput(.47,-4.98)(0,-2.2){2}{ \multiput(0,0)(2,0){3}{ \pscurve(0,0)(.05,.1)(0,.15) } }
\multiput(0,-4.4)(2,0){3}{ \psline(.48,-.42)(-1.16,-.14) }
\multiput(0,0)(2,-2.2){2}{ \multiput(-.56,-7.1)(2,0){3}{ \pscurve(-.15,.1)(-1.15,.3)(-1.85,.25)(-2.10,.37) } }
\multiput(0,0)(2,-2.2){2}{ \multiput(0,-6.6)(2,0){3}{ \pscurve(.48,-.44)(.35,-.4)(.25,-.35)(-.5,-.4)(-1.16,-.14) } }
\multiput(0,-8.8)(2,0){3}{ \psarc(0,0){.75}{0}{180} }
\multiput(2,-8.8)(2,0){3}{ \psarc(0,0){.75}{-140}{-50} }
\multiput(-.6,-8.8)(2,0){3}{ \psarc(0,0){.15}{180}{260} }
\multiput(.9,-8.8)(2,0){3}{ \psarc(0,0){.15}{180}{250} }
\multiput(1.44,-9.3)(2,0){3}{ \pscurve(0,0)(-.05,.05)(-.15,.1) }
\multiput(2.47,-9.38)(2,0){3}{ \pscurve(0,0)(.05,.1)(0,.15) }
%
\multiput(5,-.8)(0,-2.2){4}{ \psarc{<-<}(0,0){.75}{-45}{45} }
\put(5.75,-.8){ $g_*$ }
\put(5.75,-3){ $S$ }
\put(5.75,-5.2){ $g_*^{-1}$ }
\put(5.75,-7.4){ $\tau_1$ }
\psline[linewidth=.5pt]{->}(-.7,.6)(-.2,.3)
\put(-1.3,.6){ $E_2$ }
\psline[linewidth=.5pt]{->}(-.9,-.5)(-.1,-.6)
\put(-1.5,-.6){ $E_1$ }
\psline[linewidth=.5pt]{->}(-.7,-1.6)(.2,-2)
\put(-1.3,-1.6){ $B_2$ }
\psline[linewidth=.5pt]{->}(-.9,-2.7)(-.15,-2.7)
\put(-1.5,-2.8){ $B_1$ }
\end{pspicture*}
\caption{A schematic of the mapping $G_*$.}
\end{figure}

By Lemma \ref{lemma_changevari}, it is also of class $\LW$.  We now define $G : E_2^c \to B_2^c$ as
\begin{equation} \label{eq_periodicext2}
G(x) \;:=\;
\left\{\begin{array}{ll}
G_*(x), & 
\textrm{if } x \in \Omega \\
\tau_1(x), & 
\textrm{if } x \in \bigcup_{k=0}^\infty \tau_k(E_1) \\
x, & 
\textrm{if } x \in \bigcup_{k=1}^\infty \tau_k(E_2).
\end{array}\right.
\end{equation}
By the same argument as \cite[pp.\ 153-4]{Gehring}, the map $G$ is a homeomorphism.  
We also note that $G$ is ``periodic'' in the sense that, for eack $k \in \N$,
\begin{equation} \label{eq_periodic}
\big(\tau_k \circ G \circ \tau_{-k}\big)|\tau_k(\bar\B \setminus E_2) \;=\; G|\tau_k(\bar\B \setminus E_2).
\end{equation}
To see that $G$ extends $g$, consider the set $\sigma_{ab} := g^{-1}_*\big(\{a \leq x_n \leq b\}\big)$.  Its complement $\R^n \setminus \sigma_{ab}$ consists of two (connected) components.  Let $\sigma_b$ be the component containing the vector $e_n$, let $\sigma_a$ be the component containing $-e_n$, and consider the open set $N := \B \cap \sigma_b$.  By assumption, $\bar{B}_2$ lies in $\B \cap \{x_n > b\}$, so $\bar{E}_2$ lies in $N$.  From before, we have $g_* = g$ on $\B$ and $S = \tau_{-1}$ on $\{x_n > b\}$, which imply that 
$$
(S \circ g_*)(N) \;=\; 
(\tau_{-1} \circ g)(N) \;=\;
\tau_{-1}(\B \cap \{x_n > b\}) \;\subset\; \tau_{-1}(\B).
$$
By hypothesis we have $g_*^{-1} = \id$ on $\tau_{-1}(\B)$ and hence on $(S \circ g)(N)$.  It follows that
\begin{eqnarray*}
G|N \;=\; 
G_*|N \;=\;
(\tau_1 \circ g_*^{-1} \circ S \circ g_*)|N \;=\;
(\tau_1 \circ \id \circ \tau_{-1} \circ g)|N \;=\; g|N.
\end{eqnarray*}
As a result, $G$ agrees with $g$ on $N \cap E_2^c$.

Lastly, $G = \id$ holds on $\sigma_b \setminus E_2$ and $G = \tau_1$ holds on $\sigma_a$.  Using these domains for $\Omega_1$ and $\R^n \setminus \bigcup_{k=0}^\infty \tau_k(\bar\B)$ for $\Omega_2$, Lemma \ref{lemma_gluelip} implies that $G$ is locally bi-Lipschitz.  With the same choice of domains, Lemma \ref{lemma_gluing} further implies that $G \in W^{2,p}_{loc}(E_2^c;B_2^c)$.
For the case of $G^{-1}$, note that the inverse is given by 
\begin{equation} \label{eq_inversemap}
G^{-1}(x) \;=\;
\left\{\begin{array}{ll}
G_*^{-1}(x), & 
\textrm{if } x \in \Omega \setminus \tau_{-1}(B_2) \\
\tau_{-1}(x), & 
\textrm{if } x \in \bigcup_{k=0}^\infty \tau_k(E_1) \\
x, & 
\textrm{if } x \in \bigcup_{k=1}^\infty \tau_k(E_2).
\end{array}\right.
\end{equation}
Arguing similarly with $g_*(N)$ for $N$, it follows that $G^{-1} \in W^{2,p}_{loc}(B_2^c;E_2^c)$, which proves the lemma.
\end{proof}

We now observe that Lemma \ref{lemma_identity} holds true even when $B_1$ and $B_2$ are not balls.  In the preceding proof it is enough that, up to rotation, there is a slab $\{c_1 < x_n < c_2\}$ that separates $B_1$ from $B_2$.  This result, stated below, is used in Section \ref{sect_collars}.

\begin{lemma} \label{lemma_identitynonballs}
Let $p \geq 1$ and let $E_1$, $E_2$, $C_1$, and $C_2$ be Jordan domains so that $\overline{E_1} \cap \overline{E_2} = \emptyset$ and $\overline{C_1} \cap \overline{C_2} = \emptyset$.  If $g : (E_1 \cup E_2)^c \to (C_1 \cup C_2)^c$ is a homeomorphism of class $\LW$ so that 
\begin{enumerate}
\item $g(\del E_i) = \del B_i$ holds, for $i=1,2$,
\item there exists a ball $B$ containing $\bar{E_1}$ and $\bar{E_2}$ so that $g|B^c = \id|B^c$,
\item there exist a rotation $\Theta : \R^n \to \R^n$ and numbers $c_1, c_2 \in \R$, with $c_1 < c_2$, so that $\Theta(C_1) \subset \{x_n < c_1\}$ and $\Theta(C_2) \subset \{x_n > c_2\}$,
\end{enumerate} 
then there is a homeomorphism $G : E_2^c \to C_2^c$ of class $\LW$ and a neighborhood $N$ of $\del E_2$ so that $g|(N \cap E_2^c) = G|(N \cap E_2^c)$.
\end{lemma}

Though the regularity of the extension $G$ is local in nature, 
it nonetheless enjoys certain uniform properties.  We summarize them in the next lemma. 

\begin{lemma} \label{lemma_bilipoutsideball}
Let $E_1$, $E_2$, $C_1$, $C_2$, $B$, and $g$ be as in Lemma \ref{lemma_identitynonballs}.  If $G$ is the extension of $g$ as defined in Equation \eqref{eq_periodicext2}, then 
\begin{enumerate}
\item $DG \in L^\infty(E_2^c)$ and $DG^{-1} \in L^\infty(C_2^c)$;
\item the restriction $G|B^c$ is a bi-Lipschitz homeomorphism.
\end{enumerate}
\end{lemma}

\begin{proof}
From Lemma \ref{lemma_identitynonballs}, the map $G$ is already locally bi-Lipschitz.  To prove item (1), we will give a uniform bound for $L(G|K)$ over all compact subsets $K$ of $B^c$.
Let $B = \B$ and let $S$ and $g_*$ be as defined in the proof of Lemma \ref{lemma_identity}.

Again, let $\sigma_{ab} := g^{-1}_*\big(\{a \leq x_n \leq b\}\big)$ and let $\sigma_b$ and $\sigma_a$ be the (connected) components of $\R^n \setminus \sigma_{ab}$ containing the vectors $e_n$ and $-e_n$, respectively.
By Equation \eqref{eq_shearmap}, we have $S|\{x_n < a\} = \id$ and $S|\{x_n > b\} = \tau_{-1}$, which imply, respectively, the bounds $L(G|\B^c \cap \sigma_a) \leq 1$ and $L(G|\B^c \cap \sigma_b) \leq 1$.  

It remains to estimate $L(G|\B^c \cap \sigma_{ab})$.  For each $k \in \N$, the set $\sigma_{ab}^k := \sigma_{ab} \cap \tau_k(\bar\B)$ is compact, so by Lemma \ref{lemma_loclip}, the restriction $G|\sigma_{ab}^k$ is bi-Lipschitz.  Equation \eqref{eq_periodic} then implies that
$L(G|\sigma_{ab}^k) = L(G|\sigma_{ab}^1)$ holds for eack $k \in \N$.

The remaining set $\sigma_{ab} \setminus \bigcup_{k=0}^\infty \tau_k(\B)$ consists of infinitely many components, one of which is an unbounded subset $U$ of $\{x_1 < 0\}$ and the others are translates of a compact subset $K_0$ of $\sigma_{ab} \cap B(0,3)$.  
Since $g|U = \id$, it follows that
$$
G|\sigma \;=\; (\tau_1 \circ g_*^{-1} \circ S \circ g_*)|\sigma \;=\; (\tau_1 \circ S)|\sigma
$$
from which $L(G|\sigma) \leq L(S)$ follows.  By the `periodicity' of $G$ (Equation \eqref{eq_periodic}), for all $k \in \N$ we also have $L(G|\tau_k(K_0) = L(G|K_0)$.  Item (1) of the lemma follows from \cite[Thm 4.2.3.5]{EvansGariepy} and from the above estimates, where
$$
\|DG\|_{L^\infty(E_2^c)} \;\leq\; \max\left\{ 1, L(G|K_0), L(G|\sigma_{ab}^1), L(S)\right\}.
$$
Using the explicit formula in Equation \eqref{eq_inversemap}, the case of $G^{-1}$ follows similarly.

To prove item (2), let $\ell$ be any line segment that does not intersect $\B$.  The restriction $G|\ell$ is bi-Lipschitz with $L(G|\ell) \leq C$.   Since $\partial\B$ is compact, it follows from Lemma \ref{lemma_loclip} that the restriction $G|\partial\B$ is bi-Lipschitz.  

Let $x_1$ and $x_2$ be arbitrary points in $\B^c$ and let $\ell$ be the line segment in $\R^n$ which joins $x_1$ to $x_2$.  If $\ell$ crosses through $\B$, then let $y_1$ and $y_2$ be points on $\ell \cap \partial\B$, where $|x_1 - y_1| < |x_1 - y_2|$.  Since $\ell$ is a geodesic, we have the identity
$$
|x_1 - x_2| \;=\; |x_1 - y_1| + |y_1 - y_2| + |y_2 - x_2|.
$$
The Triangle inequality then implies that
\begin{eqnarray*}
|G(x_1) - G(x_2)| &\leq& 
|G(x_1) - G(y_1)| + |G(y_1) - G(y_2)| + |G(y_2) - G(x_2)| \\ &\leq&
C\,(|x_1 - y_1| + |y_2 - x_2|) + L(G|\partial\B)\,|y_1 - y_2| \\ &\leq&
\big( C + L(G|\partial\B) \big) \, (|x_1 - y_1| + |y_1 - y_2| + |y_2 - x_2|) \\ &=&
\big( C + L(G|\partial\B) \big) \, |x_1 - x_2|.
\end{eqnarray*}
Again, the argument is symmetric for $G^{-1}$, so this proves the lemma.
\end{proof}

Theorem \ref{thm_disjointcase} now follows easily from Lemma \ref{lemma_identity}, and a more general version of the theorem follows from Lemma \ref{lemma_identitynonballs}.  As in \cite[Lemma 2]{Gehring}, one takes compositions with the extension, its inverse, and a radial stretch map.

\begin{proof}[Proof of Theorem \ref{thm_disjointcase}]
By composing $g$ with linear maps, we may assume that $E_1$, $E_2$, $B_1$ and $B_2$ are subsets of $\B$, that $0 \in E_2$, and that 
$\B^c \subset g(\B^c)$.
Choose $r_1, r_2 \in (0,1)$ so that $B(0,r_1) \subset E_2$ and that $E_1 \cup E_2 \subset B(0,r_2)$.

Let $\rho : [0,\infty) \to [0,\infty)$ be a smooth increasing function so that 
$\rho\big([0,r_1]\big) = [0,r_2]$ and $\rho\big([1,\infty)\big) = [1,\infty)$.
Define a homemorphism $R : \R^n \to \R^n$ by
\begin{equation} \label{eq_radialstretch}
R(x) \;:=\; 
\left\{\begin{array}{ll}
\rho(|x|) \cdot |x|^{-1}x, & \textrm{if } x \neq 0 \\
0, & \textrm{if } x = 0.
\end{array}\right.
\end{equation}
Clearly, $R$ is of class $\LW$ and bi-Lipschitz, and maps $B(0,r_1)$ onto $B(0,r_2)$.

Putting $E_1' :=(g \circ R)(E_1)$ and $E_2' := \big( (g \circ R) (E_2^c) \big)^c$, Lemma \ref{lemma_changevari} implies that 
$$
h := g \circ R \circ g^{-1} : (E_1' \cup E_2')^c \to (B_1 \cup B_2)^c
$$ 
is also a homeomorphism of class $\LW$. Since $R|\B^c = \id|\B^c$, we further obtain
\begin{equation} \label{eq_identitydata}
h|\B^c \;=\; 
(g \circ R \circ g^{-1})|\B^c \;=\; 
\id|\B^c.
\end{equation}
So with $E_1'$ and $E_2'$ in place of $E_1$ and $E_2$, respectively, $h$ satisfies Equation \eqref{eq_identityball} and the other hypotheses of Lemma \ref{lemma_identity}.  As a result, there exists a homeomorphism $H : (E_2')^c \to B_2^c$ of class $\LW$ and a neighborhood $N'$ of $\del E_2'$ so that
$$
h|(N' \cap (E_2')^c) \;=\;
H|(N' \cap (E_2')^c).
$$
Let $G := H \circ g \circ R^{-1}$.  The open set 
$$
N \;:=\; (R \circ g^{-1})(N' \setminus (\bar{B}_1 \cup \bar{B}_2))
$$ 
contains $\del E_2$, and by Lemma \ref{lemma_changevari}, the map $G$ is of class $\LW$.  Moreover, for each $x \in N \setminus E_2$, there is a $y \in N' \setminus D_2'$ so that $x = (R \circ g^{-1})(y)$ and therefore
\begin{eqnarray*}
G(x) &=&
(H \circ g \circ R^{-1})\big((R \circ g^{-1})(y)\big) \;=\; H(y) \\ &=& 
h(y) \;=\; (g \circ R \circ g^{-1})\big((g \circ R^{-1})(x)\big) \;=\; g(x).
\end{eqnarray*}
We thereby obtain $g = G$ on $N \cap E_2^c$, as desired.
\end{proof}

\section{Extensions of Homeomorphisms of Class $\LW$ between Collars}
\label{sect_collars}

\subsection{Generalized Inversions} \label{sect_inversions}

To pass to the configurations of domains in Theorem \ref{thm_mainthm}, we will use \emph{generalized inversions}.  For fixed $a,r > 0$, these are homeomorphisms $I_{a,r} : \R^n \setminus \{0\} \to \R^n \setminus \{0\}$ of the form
$$
I_{a,r}(x) \;:=\; r^{a+1}|x|^{-(a+1)}x.
$$
Indeed, the inverse map satisfies $(I_{a,r})^{-1} = I_{1/a,\,r}$, as well as the estimate
\begin{equation} \label{eq_invinverse}
|x|^{a+1} \;=\;
\big( r^{1/a+1}|I_{a,r}(x)|^{-1/a} \big)^{a+1} \;\approx\; |I_{a,r}(x)|^{-(1/a + 1)}.
\end{equation}
For derivatives of $I_{a,r}$, an elementary computation gives
\begin{equation} \label{eq_invderiv}
|D^kI_{a,r}(x)| \;\lesssim\;
r^{a+1} |x|^{-(a+k)}
\end{equation}
and similarly, for the Jacobian determinant $JI_{a,r} := |\det(DI_{a,r})|$ we have
\begin{equation} \label{eq_invjacobian}
JI_{a,r}(x) \;\leq\;
n \, r^{n(a+1)}|x|^{-n(a+1)} \;\approx\;
|I_{a,r}(x)|^{n(a+1)/a}.
\end{equation}
If $a = 1$, then $I_{1,r}$ is conformal and maps spheres to spheres.  In general, the map $I_{a,r}$ possesses weaker properties which are sufficient for our purposes.  For instance, it preserves radial rays, or sets of the form
$\{ \lambda x : \lambda > 0 \}$ for some $x \in \R^n \setminus \{ 0 \}$.

Another property, stated below, is used in the proof of Theorem \ref{thm_mainthm} under the following hypotheses.  To begin, write $B_1 = B(t,r_1)$ and $B_2 = B(z,r_2)$, where $\bar{B}_1 \subset B_2$.  By composing with linear maps, we may assume that
\begin{itemize}
\item[(H1)]
The $x_n$-coordinate axis crosses through the points $t$ and $z$, with $t_n \leq z_n \leq 0$.  As a result, the `south poles' 
$\tau := t - r_1\vec{e}_n$ on $\bar{B}_1$ and
$\zeta := z - r_2\vec{e}_n$ on $\bar{B}_2$
satisfy $\zeta_n < \tau_n$ and $|\zeta - \tau| = \dist(\bar{B}_1, B_2^c)$.

\item[(H2)]
There exists $r \in (0,r_2)$ so that the sphere $\del B(0,r)$ is tangent to both 
$\del B_1$ and $\del B_2$, with $B(0,r) \subset B_2 \setminus B_1$.  In particular, this gives $r_1 < |t_n|$.
\end{itemize}

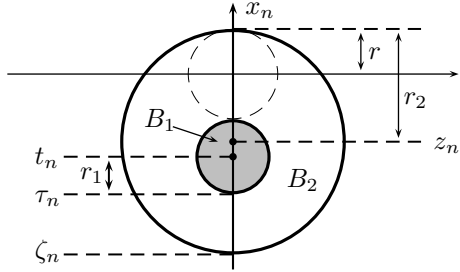
\begin{figure}[h]
\centering
\begin{pspicture*}(-3,-.7)(3,2.85)
%
\pscircle[fillcolor=lightgray,fillstyle=solid,framearc=.2,linewidth=1.2pt](0,.8){.5}
\pscircle*(0,.8){.05}
\psline[linewidth=.75pt,linestyle=dashed](-2.25,.8)(0,.8)
\psline[linewidth=.75pt,linestyle=dashed](-2.25,.3)(0,.3)
\put(-2.75,.7){ $t_n$ }
\psline[linewidth=.75pt]{<->}(-1.65,.35)(-1.65,.75)
\put(-2.15,.5){ $r_1$ }
\put(-1.3,1.2){ $B_1$ }
\psline[linewidth=.5pt]{->}(-.8,1.2)(-.2,1)
\put(-2.75,.2){ $\tau_n$ }
\psline[linewidth=.75pt,linestyle=dashed](-2.25,-.5)(0,-.5)
\put(-2.75,-.55){ $\zeta_n$ }
%
\pscircle[linewidth=1.2pt](0,1){1.5}
\pscircle*(0,1){.05}
\psline[linewidth=.75pt,linestyle=dashed](0,2.5)(2.5,2.5)
\psline[linewidth=.5pt]{<->}(1.7,1.95)(1.7,2.45)
\put(1.8,2.1){$r$}
\psline[linewidth=.5pt]{<->}(2.2,1.05)(2.2,2.45)
\put(2.15,1.5){ $r_2$ }
\psline[linewidth=.75pt,linestyle=dashed](0,1)(2.5,1)
\put(2.55,0.9){ $z_n$ }
\put(.6,.4){ $B_2$ }
\psline{->}(0,-.7)(0,2.85)
\put(.05,2.7){ $x_n$ }
\psline[linewidth=.5pt]{->}(-3,1.9)(3,1.9)
\pscircle[linewidth=.3pt,linestyle=dashed](0,1.9){.6}
%
\end{pspicture*}
\caption{A possible configuration for $B_1$, $B_2$, and $B(0,r)$.}
\end{figure}

\begin{lemma} \label{lemma_slabcond}
Let $a \in (0,1)$.  If $B_1$ and $B_2$ are balls in $\R^n$ with $\bar{B}_1 \subset B_2$ and which satisfy hypotheses (H1) and (H2), then there exist real numbers $c_1 < c_2$ so that $I_{a,r}(B_1) \subset \{ x_n < c_1 \}$ and $I_{a,r}(B_2^c) \subset \{ x_n > c_2 \}$.
\end{lemma}

The proof is a computation, and the basic idea is simple.  Though the bounded domains $I_{a,r}(B_1)$ and $I_{a,r}(B_2^c)$ may not be balls, the distance between them is still attained by the images of the `north' and `south' poles of $B_1$ and $B_2$, respectively.

\begin{proof}
Once again, let $\tau$ and $\zeta$ be the ``south poles'' of $B_1$ and $B_2$, respectively.  From Hypotheses (H1) and (H2), we have 
$$
\zeta_n \;=\; -|\zeta| \;<\; -|\tau| \;=\; \tau_n.
$$
and putting $I_ := I_{a,r}$, the image points $\tau' := I(\tau)$ and $\zeta' := I(\zeta)$ therefore satisfy
\begin{equation} \label{eq_comparepoles}
\tau'_n \;=\; -|\tau'| \;<\; -|\zeta'| \;=\; \zeta'_n.
\end{equation}

\begin{claim} \label{claim_toolow}
For all $y' \in I(B_1)$, we have $y'_n < \tau'_n$.
\end{claim}

Supposing otherwise, there exists $y \in \del B_1$ with $y \neq \tau$ and so that $y'$ 
has the same $n$th coordinate as $\tau'$. 
Let $\theta$ be the angle between the $x_n$-axis and the line crossing through $y'$ 
and $0$.  By our hypotheses, we have $t_n \leq 0$ and $0 < \theta < \frac{\pi}{2}$ and therefore $0 < \cos\theta < 1$.  From $|\tau| = r_1 - t_n$, we obtain
$$
|y'| \;=\;
\frac{|\tau'|}{\cos\theta} \;=\;
\frac{r^{a+1}|\tau|^{-a}}{\cos\theta} \;=\;
\frac{r^{a+1}}{(r_1 - t_n)^a\cos\theta}
$$
so from $|y'| = 
r^{a+1}|y|^{-a}$ and the above identity, we further obtain
\begin{equation} \label{eq_trig1}
|y| \;=\; 
r^{(a+1)/a} \left[ 
\frac{r^{a+1}}{(r_1 - t_n)^a\cos\theta} 
\right]^{-1/a} \;=\; 
(\cos\theta)^{1/a}(r_1 - t_n).
\end{equation}
On the other hand, $I$ preserves radial rays and hence angles between radial rays.  As a result, $y \in \del B_1$ (and the Law of Cosines) imply that 
\begin{eqnarray*}
r_1^2 &=& 
|y|^2 + t_n^2 - 2|y|t_n \cos\theta, \\
\textrm{so } |y| &=&
-t_n\cos\theta \,+\, \sqrt{r_1^2 - t_n^2\sin^2\theta}.
\end{eqnarray*}
From Hypothesis (H2) once again, we obtain $r_1 < |\tau_n|$ and hence
$$
|y| \;<\; 
-t_n\cos\theta \,+\, \sqrt{r_1^2 - r_1^2\sin^2\theta} \;=\;
(r_1-t_n) \cos\theta.
$$
This is in contradiction with Equation \eqref{eq_trig1}, since the inequality $\cos\theta \leq (\cos\theta)^{1/a}$ follows from $a \geq 1$.  The claim follows.

\begin{claim} \label{claim_toohigh}
For all $w' \in I(B_2^c)$, we have $\zeta'_n < w'_n$.
\end{claim}

Suppose there exists $w \in \del B_2$ so that $w \neq \zeta$ and $w'_n = \zeta'_n$. 
If $\alpha$ is the angle between $w$ and the $x_n$-axis, then a similar computation as above gives
$$
(2r_2-r)\cos^{1/a}\alpha \;=\;
|w| \;=\;
(r_2-r)\cos\alpha + \sqrt{r_2^2 - (r_2-r)^2\sin^2\theta}
$$
Computing further, we obtain $\psi(a) = r_2^2$, where $\psi : (0,\infty) \to (0,\infty)$ is given by
$$
\psi(a) \;:=\;
\big( (2r_2-r)\cos^{1/a}\alpha - (r_2-r)\cos\alpha \big)^2 + (r_2-r)^2\sin^2\alpha
$$
Clearly $\psi$ is smooth and an elementary computation shows that it attains a minimum at a unique point 
in $(0,1)$.  We observe that
$$
\psi(1) \;=\; 
r_2^2\cos^2\alpha \,+\; (r_2-r)^2\sin^2\alpha \;<\; r_2^2.
$$
Since $0 < \cos\alpha < 1$, we see that $\cos^{1/a}\alpha \to 0$ as $a \to 0$.  It follows that
$$
\lim_{a \to 0} \psi(a) \;=\;
\big( 0 + (r_2-r)\cos\alpha \big)^2 + (r_2-r)^2\sin^2\alpha \;=\; (r_2-r)^2 \;<\; r_2^2
$$
and therefore $\psi(a) < r_2^2$ holds for all $(0,1)$.  This is a contradiction, which proves Claim \ref{claim_toohigh}.
Combining both claims and Equation \eqref{eq_comparepoles}, the lemma follows.
\end{proof}

\subsection{From Doubly-Punctured Domains to Collars}

We now prove Theorem \ref{thm_mainthm}.
The argument requires several lemmas.

\begin{lemma} \label{lemma_checkinv}
Let $a > 0$ and let $D_1$, $D_2$, $B_1$, $B_2$, and $f$ be given as in Theorem \ref{thm_mainthm}.  If there exists $r > 0$ so that $\bar{B}(0,r) \subset D_2 \setminus D_1$ and $\bar{B}(0,r) \subset B_2 \setminus B_1$, and if $f(0) = 0$, then $I_{a,r} \circ f \circ I_{a,r}^{-1}$ is a homeomorphism of class $\LW$.
\end{lemma}

\begin{proof}
Since $\Omega := I_{a,r}(D_2\setminus (\bar{D}_1 \cup \{0\}))$ and $I_{a,r}(B_2\setminus (\bar{B}_1 \cup \{0\}))$ lie in $ \R^n \setminus B(0,\e)$, for some $\e > 0$,
the restricted maps $I_{a,r}^{-1}| \Omega$ 
and
$I_{a,r}| \Omega'$
are diffeomorphisms.  By Lemma \ref{lemma_changevari}, it follows that 
$g := I_{a,r} \circ f \circ I_{a,r}^{-1} : \Omega \to \Omega'$ 
is of class $\LW$.
\end{proof}

\begin{lemma} \label{lemma_invbilip}
Let $E_1$, $E_2$, $C_1$, $C_2$, $B$, and $g$ be given as in Lemma \ref{lemma_identitynonballs}, and let $G$ be given as in Equation \eqref{eq_periodicext2}.  If $0 \in E_2$, if $0 \in C_2$, and if there exists $r > 0$ so that $B = B(0,r)$, then for each $a > 0$, the map 
$$
F(x) \;:=\;
\left\{
\begin{array}{ll}
\big(I^{-1} \circ G \circ I\big)(x), & x \neq 0 \\
0, & x = 0
\end{array}
\right.
$$
is a locally bi-Lipschitz homeomorphism.
\end{lemma}

\begin{proof}
Without loss of generality, let $r=1$ and put $I = I_{a,r}$ and $b = 1/a$.  By Equation \eqref{eq_periodicext2}, we have $|G(x)| \to \infty$ as $|x| \to \infty$, so $F$ is a well-defined homeomorphism.  For each $\e > 0$, put $B_\e := B(0,\e)$.  The restrictions $I|B_\e^c$ and $I^{-1}|B_\e^c$ are diffeomorphisms, so $F|B_\e^c$ is already locally bi-Lipschitz for each $\e > 0$.  

To show that $F|B_\e$ is bi-Lipschitz, recall that $DG \in L^\infty(E_2^c)$ follows from Lemma \ref{lemma_bilipoutsideball}.
So from Equations \eqref{eq_changevari}, \eqref{eq_invinverse}, and \eqref{eq_invderiv}, it follows that, for a.e.\ $x \in I^{-1}(E_2^c)$,
\begin{eqnarray*}
|DF(x)| &\leq&
|DI^{-1}\big((G \circ I)(x)\big)| \, |DG(I(x))| \, |DI(x)| \\ &\lesssim&
\frac{\|DG\|_\infty}{|(G \circ I)(x)|^{b+1} \, |x|^{a+1}} \;\approx\;
\frac{\|DG\|_\infty \, |I(x)|^{b+1}}{|(G \circ I)(x)|^{b+1}}.
\end{eqnarray*}
Now fix $y_0 \in E_2^c$.  Putting $L := L(G^{-1}|B^c)$, for all $x \in B_\e$ we have
$$
|G(I(x)) - G(y_0)| \;\geq\;
L^{-1}\big(|I(x) - y_0|\big) \;\geq\;
L^{-1}\big(|I(x)| - |y_0|\big).
$$
Applying the triangle inequality to the right-hand side, we obtain
$$
|G(I(x))| \;\geq\;
L^{-1}\big(|I(x)| - |y_0|\big) \,-\, |G(y_0)|
$$
and taking reciprocals, we further obtain
\begin{equation}
\label{eq_nearinfin} \left\{
\begin{split}
\frac{|I(x)|}{|(G \circ I)(x)|} &\;\leq\;
\frac{L \, |I(x)|}{|I(x)| - |y_0| - L \, |G(y_0)|} \\ &\;=\;
\frac{L \, r^{a+1}}{r^{a+1} - |x|^a \, |y_0|  - |x|^a \, L \, |G(y_0)|} \;\to\; L
\end{split}
\right.
\end{equation} 
as $x \to 0$. Combining the previous estimates, for sufficiently small $\e > 0$
$$
|DF(x)| \;\lesssim\; 
\frac{\|DG\|_\infty \, |I(x)|^{b+1}}{|(G \circ I)(x)|^{b+1}} \;\lesssim\;
(2L)^{b+1}\|DG\|_\infty \;<\; \infty
$$
holds for a.e.\ $x \in B_\e$, and therefore $|DF| \in L^\infty_\loc(I^{-1}(E_2^c))$.  By \cite[Thm 4.2.3.5]{EvansGariepy}, it follows that $F$ is locally Lipschitz on $B(0,\e)$.  By symmetry, the same holds for $F^{-1}$, so $F$ is locally bi-Lipschitz on all of $I^{-1}(E_2^c)$.
\end{proof}

In the remaining proofs, we will require explicit forms of the extensions from Lemma \ref{lemma_identity} and from Theorem \ref{thm_disjointcase}.

\begin{lemma} \label{lemma_checkinvid}
Let $E_1$, $E_2$, $C_1$, $C_2$, $g$, and $B =B(0,r)$ be given as in Lemma \ref{lemma_invbilip}, let $G$ be given as in Equation \eqref{eq_periodicext2}, and let $p \in [1,n)$.  If $a < n/p-1$, then the homeomorphism $I_{a,r}^{-1} \circ G \circ I_{a,r}$ is of class $\LW$.
\end{lemma}

\begin{proof}
For convenience, we reuse the notation from the proof of Lemma \ref{lemma_invbilip}.  As before, $I|B_\e^c$ and $I^{-1}|B_\e^c$ are diffeomorphisms, so by Lemma \ref{lemma_changevari}, the map $F|B_\e^c$ is of class $\LW$.  It suffices to show that $F \in W^{2,p}_{loc}(B_\e;\R^n)$ and $F^{-1} \in W^{2,p}_{loc}(F(B_\e);B_\e)$, for each $\e > 0$.

To estimate second derivatives, we use Equations \eqref{eq_changevari}, \eqref{eq_invinverse}, \eqref{eq_invderiv}, and \eqref{eq_nearinfin} once again.  As a shorthand, put $y := I(x)$ and $z := (G \circ I)(x)$.  We then obtain
\begin{equation} \label{eq_2ndderivbd1}
\left\{ \begin{split}
|D^2F(x)| \;=\; &
|D^2(I^{-1} \circ G \circ I)(x))| \\ \;\leq\; &
|D^2I^{-1}(z)| \, |DG(y)|^2 |DI(x)|^2 \\ & \,+\,
|DI^{-1}(z)| \Big( |D^2G(y)| \, |DI(x)|^2 \,+\,
|DG(y)| \, |D^2I(x)| \Big) \\ \;\lesssim\; &
\frac{\|DG\|_\infty^2}{|z|^{b+2}|x|^{2(a+1)}} \,+\,
\frac{1}{|z|^{b+1}} \bigg( \frac{|D^2G(y)|}{|x|^{2(a+1)}} \,+\,
\frac{\|DG\|_\infty}{|x|^{a+2}} \bigg) \\ \;\lesssim\; &
\frac{|I(x)|^{2(b+1)}}{|G(I(x))|^{b+2}} \,+\,
\frac{|I(x)|^{2(b+1)}|D^2G(I(x))|}{|G(I(x))|^{b+1}} \,+\,
\frac{|I(x)|^{b+1}}{|G(I(x))|^{b+1}|x|} \\ \;\lesssim\; &
|I(x)|^b \,+\, |I(x)|^{b+1}|D^2G(I(x))| \,+\, |x|^{-1}
\end{split}\right.
\end{equation}
for a.e.\ $x \in B_\e$.  Since $p < n$ and $b = 1/a$, the function $x \mapsto |I(x)|^b = |x|^{-1}$ lies in $L^p(B_\e)$.  For the remaining term, Equations \eqref{eq_invinverse} and \eqref{eq_invjacobian} imply that
$$
1 \;=\; 
JI^{-1}\big(I(x)\big) JI(x) \;\lesssim\; 
|I^{-1}(I(x))|^{n(a+1)} JI(x) \;=\;
|I(x)|^{-n(b+1)} JI(x)
$$
so by a change of variables \cite[Thm 2.2.2]{Ziemer} and Equation \eqref{eq_invjacobian}, we have
\begin{equation} \label{eq_2ndderivbd2}
\left\{\begin{split}
\int_{B_\e} |I(x)|^{p(b+1)}|D^2G(I(x))|^p \,dx &\;\lesssim\;
\int_{B_\e} \frac{|D^2G(I(x))|^p JI(x)}{|I(x)|^{(n-p)(b+1)}} \,dx \\ &\;=\;
\int_{\B^c} \frac{|D^2G(y)|^p}{|y|^{(n-p)(b+1)}} \,dy.
\end{split}\right.
\end{equation}
For each $k \in \N$, Equation \eqref{eq_periodicext2} implies that $G|\tau_k(E_2) = \id$ and $G|\tau_k(E_1) = \tau_1$, and therefore $D^2G|\tau_k(E_1 \cup E_2) = 0$.  The rightmost integral in Equation \eqref{eq_2ndderivbd2} can therefore be restricted to the subset
$$
\Omega \;:=\; 
\B^c \setminus \bigcup_{k=1}^\infty \tau_k(E_1 \cup E_2).
$$
As defined in the proof of Lemma \ref{lemma_identity} 
the maps $g_*$, $G_*$, and $G$ satisfy
\begin{equation} \label{eq_2ndderivbd3}
|D^2G(y)| \;\lesssim\;
|D^2g_*^{-1}((S \circ g_*)(y))| + |D^2S(g_*(y))| + |D^2g_*(y)|
\end{equation}
for a.e.\ $y \in I^{-1}(E_2^c)$, and where $\lesssim$ includes the constants $L(g_*)$, $L(g_*^{-1})$, $L(S)$, and $L(\tau_1)$.  
Using the second derivative bound for $S$ (Equation \eqref{eq_derivSbd}), we obtain
$$
\int_\Omega \frac{|D^2S(g_*(y))|^p}{|y|^{(n-p)(b+1)}} \,dy \;\leq\;
\int_\Omega \frac{2c^{2p}}{|y|^{(n-p)(b+1)}} \,dy \;\lesssim\;
\int_1^\infty \frac{\rho^{n-1}}{\rho^{(n-p)(b+1)}}\,d\rho.
$$
The rightmost integral is finite, since $a < n/p-1$ implies that $b > p/(n-p)$ and
$$
(n-1) - (n-p)(b+1) \;<\; 
(n-1) - (n-p)\Big(\frac{p}{n-p} - 1\Big) \;=\;
-1.
$$
For the other terms of Equation \eqref{eq_2ndderivbd3}, Equation \eqref{eq_periodicmap} implies that $D^2g_*^{-1}(z) = 0$ for a.e.\ $z \notin \bigcup_{k=1}^\infty \tau_k(\B)$.
Since $S \circ g_*$ is locally bi-Lipschitz, we estimate 
\begin{eqnarray*}
\int_\Omega \frac{|D^2g_*^{-1}\big((S \circ g_*)(y)\big)|^p}{|y|^{(n-p)(b+1)}} \,dy &=&
\sum_{k=1}^\infty \int_{\tau_k((S \circ g_*)^{-1}(\B)) \cap \Omega} \frac{|D^2g_*^{-1}\big((S \circ g_*)(y)\big)|^p}{|y|^{(n-p)(b+1)}} \,dy 
\\ &\approx&
\sum_{k=1}^\infty \int_{g_*^{-1}(\Omega) \cap \tau_k(\B)} \frac{|D^2g_*^{-1}(z)|^p \,dz}{|(S \circ g_*)^{-1}(z)|^{(n-p)(b+1)}}
\end{eqnarray*}
Equation \eqref{eq_shearmap} implies that $|S^{-1}(y)| \geq |y|$ holds, for each $y \in \R^n$, and therefore
$$
|(S \circ g_*)^{-1}(z)| \;\geq\; 3k-1 \;>\; k
$$
holds, for each $z \in \tau_k(\B)$ and each $k \in \N$.  From the above inequalities and another change of variables, we further estimate
\begin{eqnarray*}
\int_{g_*^{-1}(\Omega) \cap \tau_k(\B)} \frac{|D^2g_*^{-1}(z)|^p}{|(S \circ g_*)^{-1}(z)|^{(n-p)(b+1)}} \,dz &\lesssim&
\int_{g_*^{-1}(\Omega) \cap \tau_k(\B)} \frac{|D^2g_*^{-1}(z)|^p}{k^{(n-p)(b+1)}} \,dz \\ &\leq&
\frac{\int_{\B \setminus (C_1 \cup C_2)} |D^2g^{-1}(z)|^p \,dz}{k^{(n-p)(b+1)}}, \\
\textrm{ so } \int_\Omega \frac{|D^2g_*^{-1}\big((S \circ g_*)(y)\big)|^p}{|y|^{(n-p)(b+1)}} \,dy &\lesssim&
\sum_{k=1}^\infty \frac{\|D^2g^{-1}\|_{L^p(\B \setminus (C_1 \cup C_2))}}{k^{(n-p)(b+1)}}.
\end{eqnarray*}
The rightmost sum is finite, since $(n-p)(b+1) > 1$ follows from the hypothesis that $a < n/p-1$.  A similar estimate gives $|y|^{(p-n)(b+1)}|D^2g_*(y)| \in L^p(B_\e)$, so by Equations \eqref{eq_2ndderivbd1}-\eqref{eq_2ndderivbd3}, we obtain  $|D^2F| \in L^p(B_\e)$, as desired.  

The same argument, with $G^{-1}$ for $G$, shows that the map $F^{-1} = I^{-1} \circ G^{-1} \circ I$ also lies in $W^{2,p}_{loc}(F(B_\e);B_\e)$.  This proves the lemma.
\end{proof}

Using the previous lemmas, we now prove the main theorem.

\begin{proof}[Proof of Theorem \ref{thm_mainthm}]
Let $a < n/p-1$ be given.
By post-composing $f$ with linear maps, we may assume that the balls $B_1$ and $B_2$ satisfy hypotheses (H1) and (H2) from Section \ref{sect_inversions}, so in particular we have $B(0,r) \subset B_2 \setminus \bar{B}_1$.  
We further assume that $B(0,r) \subset D_2 \setminus \bar{D}_1$ and $f(0) =0$.

By Lemma \ref{lemma_slabcond}, there exist $c_1 < c_2$ so that $B_1 \subset \{x_n < c_1\}$ and $B_2 \subset \{x_n > c_2 \}$.  For $I := I_{a,r}$ and $g := I \circ f \circ I^{-1}$,  Lemma \ref{lemma_checkinv} implies that $g$ is of class $\LW$.

Put $E_1 = I(D_1)$, $E_2 := I(D_2^c)^c$, $C_1 := I(B_1)$, and $C_2 := I((B_2)^c)^c$.
By Lemma \ref{lemma_identitynonballs} and the proof of Theorem \ref{thm_disjointcase}, there exists a homeomorphism $G : E_2^c \to C_2^c$ of class $\LW$ and a neighborhood $N'$ of $\del E_2$ so that 
$$
g|(N' \cap E_2^c) \;=\; G|(N' \cap E_2^c).
$$
As a result, the homeomorphism $F$, as defined in Lemma \ref{lemma_invbilip}, and the open set $N := I^{-1}(N')$, a neighborhood of $\del D_2$, therefore satisfy the identity
$$
f|(N \cap \bar{D}_2) \;=\; F|(N \cap \bar{D}_2).
$$
Recalling the proof of Theorem \ref{thm_disjointcase}, we have $G = H \circ g \circ R^{-1}$, where
\begin{itemize}
\item[(H3)] $R$ is a diffeomorphism that agrees with the identity map on $\B^c$;
\item[(H4)] $H$ is a homeomorphism of class $\LW$, as given from Lemma \ref{lemma_identitynonballs}, that agrees with $h = g \circ R \circ g^{-1}$ on the open set $(g \circ R)(N')$.
\end{itemize}
Putting $H_* := I^{-1}\circ H \circ I$ and $R_* := I^{-1}\circ R \circ I$, we rewrite
$$
F \;=\; 
I^{-1} \circ (H \circ g \circ R^{-1}) \circ I \;=\;
H_* \circ f \circ R^{-1}_*.
$$
From property (H3) and properties of $I$ and $I^{-1}$, we see that $R_*^{-1}$ is a diffeomorphism from $\R^n \setminus \{0\}$ onto itself.  In particular, for each $r > 0$ the restriction $R^{-1}_*|B(0,r)^c$ is bi-Lipschitz.  On the other hand, for sufficiently small $r > 0$ we have $R^{-1} \circ I = I$ on $B(0,r)$.  Letting $\operatorname{Id}_n$ be the $n \times n$ identity matrix,
\begin{eqnarray*}
DR^{-1}_*|B(0,r) &=&
D(I^{-1} \circ R^{-1} \circ I)|B(0,r) \;=\;
D(I^{-1} \circ I)|B(0,r) \;=\; \operatorname{Id}_n \\
D^2R^{-1}_*|B(0,r) &=&
D^2(I^{-1} \circ R^{-1} \circ I)|B(0,r) \;=\;
D^2(I^{-1} \circ I)|B(0,r) \;=\; 0.
\end{eqnarray*}
This implies that $R^{-1}_* \in W^{2,p}_{loc}(\R^n;\R^n)$ and by Lemma \ref{lemma_gluelip}, that $R^{-1}_*$ is bi-Lipschitz.  By symmetry the same holds for $R_* = I^{-1} \circ R \circ I $, so $R_*^{-1}$ is of class $\LW$.

Property (H4) and Lemma \ref{lemma_checkinvid} imply that $H_*$ is of class $\LW$.  By hypothesis, $f$ is of class $\LW$, so by Lemma \ref{lemma_changevari}, $F$ is of class $\LW$.  The theorem follows.
\end{proof}

\bibliographystyle{alpha}
\bibliography{schoenflies1}

\begin{thebibliography}{Geh67}

\bibitem[Ale24]{Alexanderhorned}
J.W. Alexander.
\newblock An example of a simply connected surface bounding a region which is
  not simply connected.
\newblock {\em Proc. Nat. Acad. Sci. U.S.A.}, 10:8--10, 1924.

\bibitem[Bro60]{Brown}
M.~Brown.
\newblock A proof of the generalized {S}choenflies theorem.
\newblock {\em Bull. Amer. Math. Soc.}, 66:74--76, 1960.

\bibitem[EG92]{EvansGariepy}
L.~C. Evans and R.~F. Gariepy.
\newblock {\em Measure theory and fine properties of functions}.
\newblock Studies in Advanced Mathematics. CRC Press, Boca Raton, FL, 1992.

\bibitem[Geh67]{Gehring}
F.~W. Gehring.
\newblock Extension theorems for quasiconformal mappings in {$n$}-space.
\newblock {\em J. Analyse Math.}, 19, 1967.

\bibitem[Hir94]{Hirsch}
M.~W. Hirsch.
\newblock {\em Differential topology}, volume~33 of {\em Graduate Texts in
  Mathematics}.
\newblock Springer-Verlag, New York, 1994.
\newblock Corrected reprint of the 1976 original.

\bibitem[HK09]{HeinonenKeith}
Juha Heinonen and Stephen Keith.
\newblock Flat forms, bi-lipschitz parametrizations, and smoothability of
  manifolds.
\newblock preprint, 2009.

\bibitem[Ker60]{Kervaire}
Michel~A. Kervaire.
\newblock A manifold which does not admit any differentiable structure.
\newblock {\em Comment. Math. Helv.}, 34:257--270, 1960.

\bibitem[KM63]{KervaireMilnor}
M.~A. Kervaire and J.~W. Milnor.
\newblock Groups of homotopy spheres. {I}.
\newblock {\em Ann. of Math. (2)}, 77:504--537, 1963.

\bibitem[LV77]{LuukkainenVaisala}
J.~Luukkainen and J.~V{\"a}is{\"a}l{\"a}.
\newblock Elements of {L}ipschitz topology.
\newblock {\em Ann. Acad. Sci. Fenn. Ser. A I Math.}, 3(1):85--122, 1977.

\bibitem[Mil56]{Milnor}
John Milnor.
\newblock On manifolds homeomorphic to the {$7$}-sphere.
\newblock {\em Ann. of Math. (2)}, 64:399--405, 1956.

\bibitem[Sul75]{Sullivan}
D.~Sullivan.
\newblock Hyperbolic geometry and homeomorphisms.
\newblock In {\em Geometric topology (Proc. Georgia Topology Conf., Athens,
  Ga., 1977)}, pages 543--555, 1975.

\bibitem[TV81]{TukiaVaisala}
P.~Tukia and J.~V{\"a}is{\"a}l{\"a}.
\newblock Lipschitz and quasiconformal approximation and extension.
\newblock {\em Ann. Acad. Sci. Fenn. Ser. A I Math.}, 6(2):303--342 (1982),
  1981.

\bibitem[V{\"a}i71]{Vaisala}
Jussi V{\"a}is{\"a}l{\"a}.
\newblock {\em Lectures on {$n$}-dimensional quasiconformal mappings}.
\newblock Lecture Notes in Mathematics, Vol. 229. Springer-Verlag, Berlin,
  1971.

\bibitem[Wil79]{Wilder}
Raymond~Louis Wilder.
\newblock {\em Topology of manifolds}, volume~32 of {\em American Mathematical
  Society Colloquium Publications}.
\newblock American Mathematical Society, Providence, R.I., 1979.
\newblock Reprint of 1963 edition.

\bibitem[Zie89]{Ziemer}
W.~Ziemer.
\newblock {\em Weakly differentiable functions}, volume 120 of {\em Graduate
  Texts in Mathematics}.
\newblock Springer-Verlag, New York, 1989.

\end{thebibliography}
\end{document}